\documentclass[11pt]{amsart}
\usepackage{amsmath, amssymb}
\usepackage{graphics,color}

\newtheorem{theorem}{Theorem}[section]

\newtheorem{lemma}[theorem]{Lemma}
\newtheorem{example}[theorem]{Example}

\def\qed{\hfill $\Box$\medskip}
\def\IC{{\mathbb C}}
\def\IR{{\mathbb R}}

\def\({\left (}
\def\){\right )}

\def\tr{{\rm tr}\,}

\def\span{{\rm span}\,}

\def\M{M_{n_1}\otimes\cdots \otimes M_{n_m}}

\def\A{A_1\otimes\cdots\otimes A_m}

\begin{document}
\openup 1.1\jot

%%%%%%%%%%%%%%%%%%%%%%%%%%%%%%%%%%%%%%%%%%%%%%%%%%%%%%%%%%%%%%%%%%%%%%%%%%%%%%%%%%%%%%%%%%%%%%%%%%%%%%%%%%%%%%%%%%%%%%%%%%%%%%%%%%%%%%%%%%%%%%%%%%%%%%%%%%%%%%%%%%%%%%%%%%%%%%%%%%%%%%%%%%%%%%%%%%%%%%%%%%%%%%%%%%%%%%%%%%%%%%%%%%%%%%%%%%%%%%%%%%%%%%%%%%%%%%%%%%%%%%%%%%%%%%%%%%%%%%%%%%%%%%%%%%%%%%%%%%%%%%%%%%%%%%%%%%%%%%%%%%%%%%%%%%%%%%%%%%%%%%%%%%%%%%%%%%%%%%%%%%%

\title[Linear maps preserving   numerical radius of tensor products of matrices]
      {Linear maps preserving   numerical radius \\of tensor products of matrices}

\author{Ajda Fo\v sner}
\author{Zejun Huang}
\author{Chi-Kwong Li}
\author{Nung-Sing Sze}
\address{Ajda Fo\v sner,
Faculty of Management, University of Primorska,
Cankarjeva 5, SI-6104 Koper, Slovenia}
\email{ajda.fosner@fm-kp.si}
\address{Zejun Huang,
Department of Applied Mathematics,
The Hong Kong Polytechnic University,
Hung Hom, Hong Kong}
\email{huangzejun@yahoo.cn}
\address{Chi-Kwong Li,
Department of Mathematics, College of William and Mary, Williamsburg, VA 23187, USA;
Department of Mathematics,
University of Hong Kong, Pokfulam, Hong Kong}
\email{ckli@math.wm.edu}
\address{  Nung-Sing Sze,
Department of Applied Mathematics,
The Hong Kong Polytechnic University,
Hung Hom, Hong Kong}
\email{raymond.sze@polyu.edu.hk}

\maketitle

\begin{abstract}
Let $m,n\ge 2$ be positive integers.  Denote by $M_m$ the set of $m\times m$ complex matrices and by $w(X)$ the numerical radius of a square matrix $X$.
Motivated by the study of operations on bipartite systems of quantum states,
we show that a linear map $\phi: M_{mn}\rightarrow M_{mn}$ satisfies
 $$w(\phi(A\otimes B))=w(A\otimes B)~~{\rm for~ all}~  A\in M_{m} ~{\rm and}~B\in M_{n}$$
if and only if there is a unitary matrix $U \in M_{mn}$ and a complex unit $\xi$ such that
$$\phi(A\otimes B) = \xi U(\varphi_1(A) \otimes \varphi_2(B))U^* \quad \hbox{ for all }
A\in M_{m} ~\hbox {\rm and}~B\in M_{n},$$
where $\varphi_k$ is the identity map or the transposition map $X \mapsto X^t$ for $k=1,2$,
and the maps $\varphi_1$ and $\varphi_2$ will be of the same type if $m,n \ge 3$.
In particular, if $m,n \ge 3$, the map corresponds to an evolution of a closed quantum
system (under a fixed unitary operator), possibly followed by a transposition.
The results are extended to multipartite systems.
\end{abstract}

\-

\noindent
{\em 2010 Math. Subj. Class.}: 15A69, 15A86, 15A60, 47A12.


%%%%%%%%%%%%%%%%%%%%%%%%%%%%%%%%%%%%%%%%%%%%%%%%%%%%%%%%%%%%%%%%%%%%%%%%%%%%%%%%%%%%%%%%%%%%%%%%%%%%%%%%%%%%%%%%%%%%%%%%%%%%%%%%%%%%%%%%%%%%%%%%%%%%%%%%%%%%%%%%%%%%%%%%%%%%%%%%%%%%%%%%%%%%%%%%%%%%%%%%%%%%%%%%%%%%%%%%%%%%%%%%%%%%%%%%%%%%%%%%%%%%%%%%%%%%%%%%%%%%%%%%%%%%%%%%%%%%%%%%%%%%%%%%%%%%%%%%%%%%%%%%%%%%%%%%%%%%%%%%%%%%%%%%%%%%%%%%%%%%%%%%%%%%%%%%%%%%%%%%%%%
\section{Introduction and preliminaries}
%%%%%%%%%%%%%%%%%%%%%%%%%%%%%%%%%%%%%%%%%%%%%%%%%%%%%%%%%%%%%%%%%%%%%%%%%%%%%%%%%%%%%%%%%%%%%%%%%%%%%%%%%%%%%%%%%%%%%%%%%%%%%%%%%%%%%%%%%%%%%%%%%%%%%%%%%%%%%%%%%%%%%%%%%%%%%%%%%%%%%%%%%%%%%%%%%%%%%%%%%%%%%%%%%%%%%%%%%%%%%%%%%%%%%%%%%%%%%%%%%%%%%%%%%%%%%%%%%%%%%%%%%%%%%%%%%%%%%%%%%%%%%%%%%%%%%%%%%%%%%%%%%%%%%%%%%%%%%%%%%%%%%%%%%%%%%%%%%%%%%%%%%%%%%%%%%%%%%%%%%%%

Let $M_n$ be the set of $n\times n$ complex matrices for any positive integer $n$.
For $A \in M_n$, define (and denote) its numerical range and numerical radius by
$$W(A)=\bigg\{u^*Au : u\in {\mathbb C}^n, u^*u = 1\bigg\}
\quad \hbox{ and } \quad
w(A)=\sup \{|\mu|: \mu\in W(A)\},$$
respectively. The study of numerical range and numerical radius
has a long history and is still under active research.
Moreover, there are many generalizations
motivated by pure and applied topics; see \cite{G,Ha,HJ}.

By the convexity of the numerical range,
$$W(A) = \{\tr(Auu^*): u \in \IC^n, u^*u = 1\} = \{\tr(AX): X \in D_n\},$$
where $D_n$ is the set of density matrices (positive semidefinite matrices with trace one)
in $M_n$.
In particular, in the study of quantum physics, if $A \in M_n$ is Hermitian corresponding to
an observable and if quantum states are represented as density matrices,
then $W(A)$ is the set of all possible measurements under the observables and
$w(A)$ is a bound for the measurement. If $A = A_1+iA_2$, where $A_1, A_2 \in M_n$
are Hermitian, then $W(A)$ is the set of the joint measurement of quantum states
under the two observables corresponding to $A_1$ and $A_2$.

Suppose $m,n\ge 2$ are positive integers. Denote by $ A\otimes B$ the tensor (Kronecker)
product of the matrices $A\in M_m$ and $B\in M_n$.
If $A$ and $B$ are observables of two quantum systems, then
$A\otimes B$ is an observable of the composite bipartite system.
Of course, a general observable on the composite system corresponds to
$C \in M_{mn}$, and observable of the form $A\otimes B$ with $A\in M_m$, $B\in M_n$
is a very small (measure zero) set.
Nevertheless, one may be able to extract useful information about the bipartite
system by focusing on the set of tensor product matrices. In particular,
in the study of linear operators $\phi: M_{mn}\rightarrow M_{mn}$ on bipartite systems,
the structure of $\phi$ can be determined by studying $\phi(A\otimes B)$ with $A\in M_m$, $B\in M_n$; see \cite{FLPS,FHLS,FHLS2,J} and their references.

In this paper, we determine the structure of linear maps
$\phi: M_{mn} \rightarrow M_{mn}$ satisfying  $w(A\otimes B) = w(\phi(A\otimes B))$ for all
$A\in M_m$ and $B\in M_n$. We show that for such a map
there is a unitary matrix $U \in M_{mn}$ and a complex unit $\xi$ such that
$$\phi(A\otimes B) = \xi U(\varphi_1(A) \otimes \varphi_2(B))U^* \quad \hbox{ for all }
A\in M_{m} \ \hbox { and } B\in M_{n},$$
where $\varphi_k$ is the identity map or the transposition map $X \mapsto X^t$ for $k=1,2$,
and the maps $\varphi_1$ and $\varphi_2$ will be of the same type if $m,n \ge 3$.
In particular, if $m,n \ge 3$, the map corresponds to an evolution of a closed quantum
system (under a fixed unitary operator), possibly followed by a transposition.

The study of linear maps on matrices or operators with some special properties
are known as preserver problems; for example, see \cite{LP} and its references.
In connection to preserver problems on bipartite quantum systems,
it is quite common that if one considers
a linear map $\phi:M_{mn} \rightarrow M_{mn}$ and imposes conditions
on $\phi(A\otimes B)$ for $A\in M_m$, $B\in M_n$, then
the partial transpose maps
$A \otimes B \mapsto A \otimes B^t$ and $A\otimes B \mapsto A^t\otimes B$
are admissible preservers.
It is interesting to note that for numerical radius preservers and numerical
range preservers in our study, if $m,n \ge 3$, then the partial transpose maps
are not allowed and that the (linear) numerical radius preserver $\phi$ on $M_{mn}$
will be of the standard form
$$X \mapsto \xi V^*XV \qquad \hbox{ or } \qquad X \mapsto \xi V^*X^tV$$
for some complex unit $\xi$
and unitary $V \in M_{mn}$. This is the first example of such results in this
line of study. It would be interesting to explore more matrix invariant or quantum
properties that the structure of preservers on  $M_{mn}$
can be completely determined by the behavior of the map on
the small class of matrices of the form $A\otimes B \in M_{mn}$
with $A\in M_m$, $B\in M_n$.

In our study, we also determine the linear map $\phi: M_{mn} \rightarrow M_{mn}$
such that
$$W(A\otimes B) = W(\phi(A\otimes B))\quad \hbox{ for all }
(A,B) \in M_m \times M_n.$$

We will denote by $X^t$ the transpose of a matrix $X\in M_n$ and
$X^*$ the conjugate transpose of a matrix $X\in M_n$. The $n\times n$
identity matrix will be denoted by $I_n$.
Let  $E_{ij}^{(n)} \in M_n$ be
the matrix whose $(i,j)$-entry is equal to one and all the others are equal to zero.
We simply write $E_{ij} = E_{ij}^{(n)}$ if the size of the matrix is clear.

We will prove our main result on bipartite systems in Section 2 and extend the results
to multi-partite systems in Section 3.

\section{Bipartite systems}
%%%%%%%%%%%%%%%%%%%%%%%%%%%%%%%%%%%%%%%%%%%%%%%%%%%%%%%%%%%%%%%%%%%%%%%%%%%%%%%%%%%%%%%%%%%%%%%%%%%%%%%%%%%%%%%%%%%%%%%%%%%%%%%%%%%%%%%%%%%%%%%%%%%%%%%%%%%%%%%%%%%%%%%%%%%%%%%%%%%%%%%%%%%%%%%%%%%%%%%%%%%%%%%%%%%%%%%%%%%%%%%%%%%%%%%%%%%%%%%%%%%%%%%%%%%%%%%%%%%%%%%%%%%%%%%%%%%%%%%%%%%%%%%%%%%%%%%%%%%%%%%%%%%%%%%%%%%%%%%%%%%%%%%%%%%%%%%%%%%%%%%%%%%%%%%%%%%%%%%%%%%

%%%%%%%%%%%%%%%%%%%%%%%%%%%%%%%%%%%%%%%%%%%%%%%%%%%%%%%%%%%%%%%%%%%%%%%%%%%%%%%%%%%%%%%%%%%%%%%%%%%%%%%%%%%%%%%%%%%%%%%%%%%%%%%%%%%%%%%%%%%%%%%%%%%%%%%%%%%%%%%%%%%%%%%%%%%%%%%%%%%%%%%%%%%%%%%%%%%%%%%%%%%%%%%%%%%%%%%%%%%%%%%%%%%%%%%%%%%%%%%%%%%%%%%%%%%%%%%%%%%%%%%%%%%%%%%%%%%%%%%%%%%%%%%%%%%%%%%%%%%%%%%%%%%%%%%%%%%%%%%%%%%%%%%%%%%%%%%%%%%%%%%%%%%%%%%%%%%%%%%%%%%

%%%%%%%%%%%%%%%%%%%%%%%%%%%%%%%%%%%%%%%%%%%%%%%%%%%%%%%%%%%%%%%%%%%%%%%%%%%%%%%%%%%%%%%%%%%%%%%%%%%%%%%%%%%%%%%%%%%%%%%%%%%%%%%%%%%%%%%%%%%%%%%%%%%%%%%%%%%%%%%%%%%%%%%%%%%%%%%%%%%%%%%%%%%%%%%%%%%%%%%%%%%%%%%%%%%%%%%%%%%%%%%%%%%%%%%%%%%%%%%%%%%%%%%%%%%%%%%%%%%%%%%%%%%%%%%%%%%%%%%%%%%%%%%%%%%%%%%%%%%%%%%%%%%%%%%%%%%%%%%%%%%%%%%%%%%%%%%%%%%%%%%%%%%%%%%%%%%%%%%%%%%

The following example is useful in our discussion.

\begin{example}\em \label{Ex1}  Suppose $m, n \ge 3$.
Let
$A = X \oplus O_{m-3}$ and  $B = X \oplus O_{n-3}$ with
$X=\begin{pmatrix}0 & 2 & 0 \cr 0 & 0 & 1 \cr 0 & 0 & 0 \cr\end{pmatrix}$.
Then $A \otimes B$ is unitarily similar to
$$\begin{pmatrix}0 & 1  \cr 0 & 0 \end{pmatrix}
\oplus \begin{pmatrix}0 & 1  \cr 0 & 0 \end{pmatrix}
\oplus \begin{pmatrix}0 & 2 & 0 \cr 0 & 0 & 1/2 \cr 0 & 0 & 0 \cr\end{pmatrix}
\oplus O_{mn-7},$$
and $A \otimes B^t$ is unitarily similar to
$$\begin{pmatrix}0 & 1/2  \cr 0 & 0 \end{pmatrix}
\oplus \begin{pmatrix}0 & 2  \cr 0 & 0 \end{pmatrix}
\oplus \begin{pmatrix}0 & 1 & 0 \cr 0 & 0 & 1 \cr 0 & 0 & 0 \cr\end{pmatrix}
\oplus O_{mn-7}.$$
One readily checks (see also \cite{LT})
that  $W(A\otimes B)$ and $W(A\otimes B^t) = W(A^t\otimes B)$ are
circular disks centered at the origin with radii $w(A\otimes B)$ and
$w(A\otimes B^t)$, respectively. Moreover, we have
\begin{eqnarray*}
w(A\otimes B) &=&
\lambda_{\max} (A\otimes B + (A\otimes B)^t)/2 = \sqrt{4.25} \\
&>& 2.0000 = \lambda_{\max} (A\otimes B^t + (A\otimes B^t)^t)/2 = w(A\otimes B^t).
\end{eqnarray*}
\vskip -.2in \qed
\end{example}

\bigskip
In the following, we first determine the structure of linear preservers of numerical range using the above example
and the results in \cite{FHLS}.

\iffalse
In this   section we will  use the trace function $\tr:M_k\to \mathbb C$,
i.e., $\tr (X)$ is the sum of the   of a matrix $X$. Note that $\tr$ is a similarity-invariant linear functional.
\fi
\begin{theorem}
\label{T2}
The following are equivalent for a linear map $\phi: M_{mn} \rightarrow M_{mn}$.
\begin{enumerate}
\item[{\rm (a)}] $W(\phi(A\otimes B)) = W(A\otimes B)$ for any $A \in M_m$ and $B \in M_n$.

\item[{\rm (b)}] There is a unitary matrix $U \in M_{mn}$ such that
$$\phi(A\otimes B) = U(\varphi_1(A) \otimes \varphi_2(B))U^*~  for~all~  A\in M_m ~and~ B\in M_n,$$
where $\varphi_k$ is the identity map or the transposition map $X \mapsto X^t$ for $k=1,2$,
and the maps $\varphi_1$ and $\varphi_2$ will be of the same type if $m,n \ge 3$.
\end{enumerate}
\end{theorem}

\begin{proof}
Suppose (b) holds. If $m, n \ge 3$, then the map has the form
$C \mapsto UCU^*$ or $C \mapsto UC^tU^*$. Thus, the condition (a) holds.
If $m = 2$, then $A^t$ and $A$ are unitarily similar for every
$A \in M_2$. So, $W(A \otimes B) = W(A^t \otimes B)$ for any $B\in M_n$. Hence, the condition (a) holds.
Similarly, if $n = 2$, then (a) holds.

Conversely, suppose that  $W(\phi(A\otimes B)) = W(A\otimes B)$ for all $A \in M_m$ and $B \in M_n$.
Assume for the moment that $A \otimes B\in M_{mn}$ is a Hermitian matrix. Then
$$W(\phi(A\otimes B)) = W(A\otimes B) \subseteq \IR.$$
This yields that $\phi(A\otimes B)$ is a Hermitian matrix, as well. Thus, $\phi$ maps Hermitian matrices to Hermitian matrices and preserves numerical radius, which is equivalent to spectral radius for Hermitian matrices. By Theorem 3.3 in \cite{FHLS}, we conclude that $\phi$ has the asserted form on Hermitian matrices and, hence, on all matrices in $M_{mn}$. However, if $m, n \ge 3$,
then, by Example \ref{Ex1}, neither the map $A\otimes B \mapsto A\otimes B^t$
nor the map $A \otimes B \mapsto A^t\otimes B$ will preserve the numerical range.
So,  the last statement about $\varphi_1$ and $\varphi_2$ holds.
\end{proof}

Next, we turn to linear preservers of the numerical radius. We need the following (well-known)
lemma to prove our result. We include a short proof for the sake of the completeness.

\begin{lemma}\label{le3}
Let $A\in M_n$ with $w(A)=|x^*Ax|=1$ for some unit $x\in \mathbf{C}^n$. Then for any unitary $U\in M_n$ with $x$ being its first column, there exists some $y\in \mathbf{C}^{n-1}$ such that
\begin{equation}\label{eq1}
U^*AU=x^*Ax\begin{pmatrix}1&y^*\\-y&* \end{pmatrix}.
\end{equation}
\end{lemma}
\begin{proof}
Write
 $(x^*Ax)^{-1} A=G+iH$ with $G$ and $H$ Hermitian. Then the largest eigenvalue of $ G$  is 1 with $x$ as its corresponding eigenvector and  $U^* GU =[1]\oplus G_1$ for some $G_1\in M_{n-1}$. Moreover, the (1,1)-entry of $iU^* HU$ is 0 since $w(A)=1$.  Since $U^* HU$ is a Hermitian matrix we have $iU^* HU =\begin{pmatrix}0 & y^* \cr -y & * \cr\end{pmatrix}$. Thus, $U^*AU$ has
the claimed form.
\end{proof}

%%%%%%%%%%%%%%%%%%%%%%%%%%%%%%%%%%%%%%%%%%%%%%%%%%%%%%%%%%%%%%%%%%%%%%%%%%%%%%%%%%%%%%%%%%%%%%%%%%%%%%%%%%%%%%%%%%%%%%%%%%%%%%%%%%%%%%%%%%%%%%%%%%%%%%%%%%%%%%%%%%%%%%%%%%%%%%%%%%%%%%%%%%%%%%%%%%%%%%%%%%%%%%%%%%%%%%%%%%%%%%%%%%%%%%%%%%%%%%%%%%%%%%%%%%%%%%%%%%%%%%%%%%%%%%%%%%%%%%%%%%%%%%%%%%%%%%%%%%%%%%%%%%%%%%%%%%%%%%%%%%%%%%%%%%%%%%%%%%%%%%%%%%%%%%%%%%%%%%%%%%%

\begin{theorem}
\label{T3}
The following are equivalent for a linear map $\phi: M_{mn} \rightarrow M_{mn}$.
\begin{enumerate}
\item[{\rm (a)}] $w(\phi(A\otimes B)) = w(A\otimes B)$ for any $A \in M_m$ and $B \in M_n$.

\item[{\rm (b)}] There is a unitary matrix $U \in M_{mn}$ and a complex unit $\xi$ such that
$$\phi(A\otimes B) = \xi U(\varphi_1(A) \otimes \varphi_2(B))U^*~  for~all~  A\in M_m ~and~ B\in M_n,$$
where $\varphi_k$ is the identity map or the transposition map $X \mapsto X^t$ for $k=1,2$,
and the maps $\varphi_1$ and $\varphi_2$ will be of the same type if $m,n \ge 3$.
\end{enumerate}
\end{theorem}

\begin{proof}
The implication (b) $\Rightarrow$ (a) can be verified readily. Now, suppose (a) holds and let $B_{ij} = \phi(E_{ii}\otimes E_{jj})$ for $1 \le i \le m, 1 \le j \le n$.
According to the assumptions, for all $1 \le i \le m, 1 \le j \le n$, there is a unit vector $u_{ij}\in \mathbf{C}^{mn}$ and a complex unit $\xi_{ij}$
such that $u_{ij}^*B_{ij}u_{ij}=\xi_{ij}$. We will first show that there exists a unitary matrix $U\in M_{mn}$ and complex units $\xi_{ij}$
such that $$B_{ij}= \xi_{ij} U(E_{ii}\otimes E_{jj})U^*, \qquad 1 \le i \le m, 1 \le j \le n.$$
We divide the proof into several claims.

%%%%%%%%%%%%%%%%%%%%%%%%%%%%%%%%%%%%%%%%%%%%%%%%%%%%%%%%%%%%%%%%%%%%%%%%%%%%%%%%%%%%%%%%%%%%%%%%%%%%%%%%%%%%%%%%%%%%%%%%%%%%%%%%%%%%%%%%%%%%%%%%%%%%%%%%%%%%%%%%%%%%%%%%%%%%%%%%%%%%%

\vspace{0,4cm}

\noindent
{\bf Claim 1.}
Suppose  $(i,j) \ne (r,s)$. If $u \in \IC^{mn}$ is a unit vector such that $|u^*B_{ij}u| = 1$, then  $B_{rs}u = 0$.

\vspace{0,2cm}

\it Proof. \rm Suppose $u$ is a unit vector such that $ u^*B_{ij}u  = e^{i\theta}$. Let $\xi = u^*B_{rs}u$.
We first consider the case if $i = r$.
For any $\mu \in \IC$ with $|\mu| \le 1$,
\begin{equation} \label{ck-1}
1 = w(E_{ii} \otimes (E_{jj} + \mu E_{ss})) =  w(B_{ij} + \mu B_{rs})
\ge  |u^*(B_{ij} + \mu B_{rs})u|
= |e^{i\theta} + \mu \xi|.
\end{equation}
Then we must have $\xi = 0$. Otherwise, $|e^{i\theta} + \mu \xi |
= |e^{i\theta} + e^{i\theta} |\xi|| > 1$
if one chooses $\mu = e^{i\theta} \bar \xi / |\xi|$.
Suppose $\xi = 0$. Then the inequality in (\ref{ck-1})
become equality and  $w(B_{ij} + \mu B_{rs}) = |u^*(B_{ij} +\mu B_{rs})u| = 1$.
By Lemma \ref{le3}, there is $y_\mu \in \IC^{mn-1}$ such that
$$U^*B_{ij}U + \mu U^*B_{rs} U = U^*(B_{ij} + \mu B_{rs})U = e^{i\theta}
\begin{pmatrix} 1 & y_\mu^* \cr -y_\mu & * \end{pmatrix},$$
where $U$ is a unitary matrix with $u$ as its first column.
Since the above equation holds for any $\mu \in \IC$ with $|\mu| \le 1$,
the matrix $U^*B_{rs}U$ must have the form $\begin{pmatrix} 0 & 0 \cr 0 & * \end{pmatrix}$. 
So, $U^*B_{rs}U$ is a matrix with zeros in its first column and row,
or equivalently,  $B_{rs} u = 0$, as desired.
Similarly, we can prove the case if $j = s$.

Now we consider the case when $i\ne r$ and $j \ne s$.
By the previous argument, we have $B_{is} u = B_{rj} u = 0$.
Then for any $\mu \in \IC$ with $|\mu| \le 1$,
\begin{eqnarray*} %\label{ck-2}
1 &=& w( (E_{ii} + E_{rr}) \otimes (E_{jj} + \mu E_{ss})) \cr
&=&  w(B_{ij} + B_{rj} + \mu (B_{is} + B_{rs})) \cr
&\ge&  |u^*(B_{ij} + B_{rj} + \mu (B_{is} + B_{rs}))u| \cr
&=& |u^*B_{ij}u + \mu u^*B_{rs}u|\cr
&=& |e^{i\theta} + \mu \xi|.
\end{eqnarray*}
It follows that $\xi = u^*B_{rs} u = 0$ and hence
$w(B_{ij} + B_{rj} + \mu (B_{is} + B_{rs})) = |u^*B_{ij} u| = 1$.
By Lemma \ref{le3}, we conclude that $(B_{is} + B_{rs})u = 0$ and thus, $B_{rs} u = 0$.
\qed

\vspace{0,4cm}

\noindent
{\bf Claim 2.}
Suppose $(i,j) \ne (r,s)$. If $u_{ij}, u_{rs} \in \IC^{mn}$
are two unit vectors such that
$|u_{ij}^*B_{ij} u_{ij}| = |u_{rs}^*B_{rs} u_{rs}| = 1$,
then $u_{ij}^* u_{rs} = 0$.

\vspace{0,2cm}

\it Proof. \rm Suppose $u_{rs} =  \alpha u_{ij} + \beta v$
with $\alpha = u_{ij}^* u_{rs}$ and $\beta = v^*u_{rs}$,
where $v$ is a unit vector orthogonal to $u_{ij}$.
Notice that $|\alpha|^2 + |\beta|^2 = 1$.
By Claim 1, $u_{ij}^*B_{rs} = 0$ and $B_{rs} u_{ij} = 0$ and so
$$1 = |u_{rs}^*B_{rs} u_{rs} | =  |\beta|^2 |v^* B_{rs} v| \le |\beta|^2 w(B_{rs}) = |\beta|^2 \le 1.$$
Thus, $|\beta| = 1$ and hence $u_{ij}^* u_{rs} = \alpha = 0$.
\qed

\vspace{0,4cm}

\noindent
{\bf Claim 3.}
Let  $U=[u_{11}~\cdots~u_{1n}~u_{21}\cdots~u_{2n}~\cdots~u_{m1}~\cdots~u_{mn}].$  Then $U^*U = I_{mn}$ and
$U^*B_{ij}U = \xi_{ij} (E_{ii}\otimes E_{jj})$
for all $1\le i\le m$ and $1\le j\le n$.

\vspace{0,2cm}
\it Proof. \rm By Claim 2, $\{u_{ij}: 1\le i\le m, 1\le j\le n\}$ forms an orthonormal basis. Thus, $U^*U = I_{mn}$.
Next by Claim 1, $u_{rs}^* B_{ij} u_{k\ell} = 0$ for all $(r,s)$ and $(k,\ell)$,
except the case when $(r,s) = (k,\ell) = (i,j)$.
Therefore, the result follows.
\qed

According to our assumptions and by Claims $1,2,3$, we see that up to some unitary similarity
$$\phi(E_{ii}\otimes E_{jj})= \xi_{ij} (E_{ii}\otimes E_{jj})$$
for $1 \le i \le m, 1 \le j \le n$ and some complex units $\xi_{ij}$. Now, for any unitary $X \in M_m$, using the same arguments as above, there exists some unitary $U_X$ and some complex units $\mu_{ij}$ such that
$$\phi(XE_{ii}X^* \otimes E _{jj}) = \mu_{ij} U_X(E_{ii} \otimes E _{jj})U_X^*$$
for all $1\le i \le m$ and $1\le j\le n$.
%We see that $\phi(XE_{ii}X^*\otimes E_{jj})$ is a rank one matrix with numerical radius one.
%If  $\gamma > 0$, then
%$$w(\phi((XE_{ii}X^*+\gamma I_m)\otimes E_{jj})) = 1+\gamma.$$
So, $\phi(XE_{ii}X^*\otimes E_{jj})$ is a unit multiple of rank one Hermitian matrix with numerical radius one. Thus,  
$\phi(XE_{ii}X^*\otimes E_{jj}) = \mu_{ij} xx^*$ for some unit vector $x\in \IC^{mn}$.
Note also that $\phi(I_m \otimes E_{jj}) = D \otimes E_{jj}$ for some diagonal unitary matrix $D$.
If  $\gamma > 0$, then
$$w(\phi((XE_{ii}X^*+\gamma I_m)\otimes E_{jj})) = 1+\gamma.$$
Furthermore, there exists a unit vector $u\in \IC^{mn}$ such that $|u^*x| = 1$ and $|u^* (D\otimes E_{jj}) u| = 1$.
From the second equality, $u$ must have the from $u = \hat u \otimes e_j$ for some unit vector $\hat u\in \IC^m$. Therefore, $|u^*x| = 1$ implies $x = \hat x \otimes e_j$ for some unit vector $\hat x \in \IC^m$. Thus, $\phi(XE_{ii}X^*\otimes E_{jj})$ has the form $R_{i,X} \otimes E_{jj}$ for some $R_{i,X} \in M_m$.
Since this is true for any $1\le i \le m$ and unitary $X\in M_m$,
we have $$\phi(A\otimes E_{jj}) = \varphi_j(A)\otimes E_{jj}$$
for all matrices $A\in M_m$ and some linear map $\varphi_j$. Clearly, $\varphi_j$ preserves numerical radius and, hence, has the form
$$A \mapsto \xi_j W_jAW_j^* \quad \hbox{ or } \quad A \mapsto \xi_j W_jA^tW_j^*$$ for some complex unit $\xi_j$ and unitary $W_j \in M_m$.
In particular, $\varphi_j(I_m)=\xi_j I_m$ and $\phi(I_{mn}) = I_m\otimes D$ for some diagonal matrix $D\in M_n$. Using the same arguments as above, we can show that $$\phi(E_{ii}\otimes B) = E_{ii}\otimes \varphi_i(B)$$ for all  matrices $B\in M_n$ and some linear map $\varphi_i$ of the form $$B \mapsto \tilde{\xi_i} \tilde{W_i}B\tilde{W_i}^* \quad \hbox{ or } \quad B \mapsto \tilde{\xi_i}\tilde{W_i}B^t\tilde{W_i}^*,$$ where $\tilde{\xi_i}$ is a complex unit and $\tilde{W_i} \in M_n$ a unitary matrix. Therefore, we have $\varphi_i(I_n)=\tilde{\xi_i} I_n$ and $\phi(I_{mn}) = \tilde{D}\otimes I_n$ for some diagonal matrix $\tilde{D}\in M_m$.
Since $\phi(I_{mn}) = \tilde D\otimes I_n =I_m\otimes D$, we conclude that  $\phi(I_{mn})=\xi I_{mn}$ for some complex unit $\xi$. For the sake of the simplicity,
let us assume that $\phi(I_{mn})= I_{mn}$. Then $\phi(E_{ii} \otimes E_{jj}) = E_{ii} \otimes E_{jj}$ for all $1\le i \le m$, $1\le j \le n$.

For any Hermitian matrices $A\in M_m$ and $B\in M_n$, suppose their spectral decompositions are $A=XD_1X^*$ and $B=YD_2Y^*$. Repeating the above argument and using the assumption $\phi(I_{mn})= I_{mn}$, one sees that there exists a unitary matrix $U_{X,Y}$ such that
$$\phi(XE_{ii}X\otimes YE_{jj}Y^*)=U_{X,Y}(XE_{ii}X^*\otimes YE_{jj}Y^*)U_{X,Y}^*,\qquad 1\le i \le m,  1\le j \le n,$$
and, hence, $\phi(A\otimes B)=U_{X,Y}(A\otimes B)U_{X,Y}^*$.
So, $\phi$ maps Hermitian matrices to Hermitian matrices and preserves numerical range on the tensor product of Hermitian matrices. Thus, by the same argument as in the proof of Theorem \ref{T2}, $\phi$ has the asserted form on Hermitian matrices
and, hence, on all matrices in $M_{mn}$.
If $m, n \ge 3$, we can use Example \ref{Ex1} to conclude
that $\varphi_1$ and $\varphi_2$ should both be the identity map, or both be the transpose map. The proof is completed.
\end{proof}

\section{Multipartite systems}
%%%%%%%%%%%%%%%%%%%%%%%%%%%%%%%%%%%%%%%%%%%%%%%%%%%%%%%%%%%%%%%%%%%%%%%%%%%%%%%%%%%%%%%%%%%%%%%%%%%%%%%%%%%%%%%%%%%%%%%%%%%%%%%%%%%%%%%%%%%%%%%%%%%%%%%%%%%%%%%%%%%%%%%%%%%%%%%%%%%%%%%%%%%%%%%%%%%%%%%%%%%%%%%%%%%%%%%%%%%%%%%%%%%%%%%%%%%%%%%%%%%%%%%%%%%%%%%%%%%%%%%%%%%%%%%%%%%%%%%%%%%%%%%%%%%%%%%%%%%%%%%%%%%%%%%%%%%%%%%%%%%%%%%%%%%%%%%%%%%%%%%%%%%%%%%%%%%%%%%%%%%

In this section we  extend   Theorem \ref{T2} and Theorem \ref{T3} to multipartite systems $\M$, $m\ge 2$.

%%%%%%%%%%%%%%%%%%%%%%%%%%%%%%%%%%%%%%%%%%%%%%%%%%%%%%%%%%%%%%%%%%%%%%%%%%%%%%%%%%%%%%%%%%%%%%%%%%%%%%%%%%%%%%%%%%%%%%%%%%%%%%%%%%%%%%%%%%%%%%%%%%%%%%%%%%%%%%%%%%%%%%%%%%%%%%%%%%%%%%%%%%%%%%%%%%%%%%%%%%%%%%%%%%%%%%%%%%%%%%%%%%%%%%%%%%%%%%%%%%%%%%%%%%%%%%%%%%%%%%%%%%%%%%%%%%%%%%%%%%%%%%%%%%%%%%%%%%%%%%%%%%%%%%%%%%%%%%%%%%%%%%%%%%%%%%%%%%%%%%%%%%%%%%%%%%%%%%%%%%%

\begin{theorem}
\label{T5}
Let $n_1, \dots, n_m \ge 2$ be positive integers and $N = \prod_{j=1}^m n_j$.
The following are equivalent for a linear map $\phi: M_N \rightarrow M_N$.
\begin{enumerate}
\item[{\rm (a)}] $W(\phi(\A)) = W(\A)$ for any $(A_1, \dots , A_m) \in M_{n_1} \times \cdots \times M_{n_m}$.

\item[{\rm (b)}] There is a unitary matrix $U \in M_{N}$ such that
\begin{equation}\label{eq2}
\phi(\A) = U(\varphi_1(A_1)\otimes \cdots\otimes\varphi_m(A_m))U^*
\end{equation}
for all $(A_1, \dots , A_m) \in M_{n_1} \times \cdots \times M_{n_m}$, where $\varphi_k$
is the identity map or the transposition map $X \mapsto X^t$ for $k=1,\ldots,m$,
and the maps $\varphi_j$
are of the same type for those $j$'s such that $n_j\geq 3$.
\end{enumerate}
\end{theorem}

\begin{proof}
The sufficient part is clear.  For the converse,  as in the proof of Theorem \ref{T2}, consider
Hermitian matrix $A = A_1 \otimes \cdots \otimes A_m $ with $A_j \in H_{n_j}$ for $j = 1, \dots, m$.
By \cite[Theorem 3.4]{FHLS}, $\phi$ has the asserted form on Hermitian matrices and, hence, on all matrices in $\M$.

However, if $n_i, n_j \ge 3$ with $i<j$,  let
$A_i = X \oplus O_{n_i-3}$ and  $A_j = X \oplus O_{n_j-3}$, where $X$ is defined as in
Example \ref{Ex1}, and
$A_k=E_{11}\in M_{n_k}$ for $k\ne i,j$.
Then $w(\A) = \sqrt{4.25}$
and $w(A_1\otimes\cdots\otimes A_{j-1}\otimes A_j^t\otimes A_{j+1}\otimes\cdots\otimes A_m)
=  2.0000$.
Thus,
$$W(\A) \ne W(A_1\otimes\cdots\otimes A_{j-1}\otimes A_j^t\otimes A_{j+1}\otimes\cdots\otimes A_m),$$
and we see that the last statement about $\varphi_k$ holds.
\end{proof}

%%%%%%%%%%%%%%%%%%%%%%%%%%%%%%%%%%%%%%%%%%%%%%%%%%%%%%%%%%%%%%%%%%%%%%%%%%%%%%%%%%%%%%%%%%%%%%%%%%%%%%%%%%%%%%%%%%%%%%%%%%%%%%%%%%%%%%%%%%%%%%%%%%%%%%%%%%%%%%%%%%%%%%%%%%%%%%%%%%%%%%%%%%%%%%%%%%%%%%%%%%%%%%%%%%%%%%%%%%%%%%%%%%%%%%%%%%%%%%%%%%%%%%%%%%%%%%%%%%%%%%%%%%%%%%%%%%%%%%%%%%%%%%%%%%%%%%%%%%%%%%%%%%%%%%%%%%%%%%%%%%%%%%%%%%%%%%%%%%%%%%%%%%%%%%%%%%%%%%%%%%%

\begin{theorem}
\label{T6}
Let $n_1, \dots, n_m \ge 2$ be positive integers and $N = \prod_{j=1}^m n_j$.
The following are equivalent for a linear map $\phi: M_N \rightarrow M_N$.
\begin{enumerate}
\item[{\rm (a)}] $w(\phi(\A)) = w(\A)$ for any $(A_1, \dots , A_m) \in M_{n_1} \times \cdots \times M_{n_m}$.

\item[{\rm (b)}] There is a unitary matrix $U \in M_N$ and a complex unit $\xi$ such that
\begin{equation}\label{eq2b}
\phi(\A) =\xi U(\varphi_1(A_1)\otimes \cdots\otimes\varphi_m(A_m))U^*
\end{equation}
for all $(A_1, \dots , A_m) \in M_{n_1} \times \cdots \times M_{n_m}$, where $\varphi_k$ is the identity map or the transposition map $X \mapsto X^t$ for $k=1,\ldots,m$,
and the maps $\varphi_j$ are of the same type for those $j$'s such that $n_j\geq 3$.
\end{enumerate}
\end{theorem}

%%%%%%%%%%%%%%%%%%%%%%%%%%%%%%%%%%%%%%%%%%%%%%%%%%%%%%%%%%%%%%%%%%%%%%%%%%%%%%%%%%%%%%%%%%%%%%%%%%%%%%%%%%%%%%%%%%%%%%%%%%%%%%%%%%%%%%%%%%%%%%%%%%%%%%%%%%%%%%%%%%%%%%%%%%%%%%%%%%%%%%%%%%%%%%%%%%%%%%%%%%%%%%%%%%%%%%%%%%%%%%%%%%%%%%%%%%%%%%%%%%%%%%%%%%%%%%%%%%%%%%%%%%%%%%%%%%%%%%%%%%%%%%%%%%%%%%%%%%%%%%%%%%%%%%%%%%%%%%%%%%%%%%%%%%%%%%%%%%%%%%%%%%%%%%%%%%%%%%%%%%%

\begin{proof}
The sufficiency part is clear. We verify the necessity. First of all, one can use similar arguments as in the proof of Theorem \ref{T3}
(see Claim 1, Claim 2, and Claim 3) to show that up to some unitary similarity,
\begin{equation*}
\label{E}
\phi(E_{i_1i_1} \otimes E _{i_2i_2}\otimes \cdots \otimes E_{i_mi_m}) =
\xi_{i_1i_2\cdots i_m} (E_{i_1i_1} \otimes E _{i_2i_2}\otimes \cdots \otimes E_{i_mi_m})
\end{equation*}
for all $1\le i_k\le n_k$ with $1\le k\le m$ and some complex units $\xi_{i_1i_2\cdots i_m}$. Below we give the details
of the proof for the case $M_{n_1} \otimes M_{n_2} \otimes M_{n_3}$. One readily extends the arguments to the general case.

%%%%%%%%%%%%%%%%%%%%%%%%%%%%%%%%%%%%%%%%%%%%%%%%%%%%%%%%%%%%%%%%%%%%%%%%%%%%%%%%%%%%%%%%%%%%%%%%%%%%%%%%%%%%%%%%%%%%%%%%%%%%%%%%%%%%%%%%%%%%%%%%%%%%%%%%%%%%%%%%%%%%%%%%%%%%%%%%%%%%%

\vspace{0,2cm}

Denote $B_{ijk} = \phi(E_{ii} \otimes E_{jj} \otimes E_{kk})$ for $1 \le i \le n_1, 1 \le j \le n_2, 1 \le k \le n_3$ and let $N=n_1n_2n_3$.
According to the assumptions, for any $1 \le i \le n_1, 1 \le j \le n_2, 1 \le k \le n_3$, there is a unit vector $u_{ijk}\in \mathbf{C}^N$
and a complex unit $\xi_{ijk}$ such that $u_{ijk}^*B_{ijk}u_{ijk}=\xi_{ijk}$.

%%%%%%%%%%%%%%%%%%%%%%%%%%%%%%%%%%%%%%%%%%%%%%%%%%%%%%%%%%%%%%%%%%%%%%%%%%%%%%%%%%%%%%%%%%%%%%%%%%%%%%%%%%%%%%%%%%%%%%%%%%%%%%%%%%%%%%%%%%%%%%%%%%%%%%%%%%%%%%%%%%%%%%%%%%%%%%%%%%%%%

\vspace{0,4cm}

\noindent
{\bf Claim 4.}
Suppose  $(i,j,k) \ne (r,s,t)$. If $u \in \IC^N$ is a unit vector such that $|u^*B_{ijk}u| = 1$, then  $B_{rst}u = 0$.

\vspace{0,2cm}

\it Proof. \rm Suppose $u$ is a unit vector such that $ u^*B_{ijk}u  = e^{i\theta}$. Let $\xi = u^*B_{rst}u$.
First, assume that $\delta_{ir}+\delta_{js}+\delta_{kt}=2$, where $\delta_{ab}$ equals to 1 when $a=b$ and zero otherwise. In other words, we assume that exactly two of the three
sets $\{i,r\}, \{j,s\}, \{k,t\}$ are singletons.
Without loss of generality, assume that $i = r$, $j = s$, and $k \ne t$.
For any $\mu \in \IC$ with $|\mu| \le 1$,
$$1 =  w(B_{ijk} + \mu B_{rst})
\ge  |u^*(B_{ijk} + \mu B_{rst})u|
= |e^{i\theta} + \mu \xi|.$$
Then we must have $\xi = 0$. Furthermore,
$w(B_{ijk} + \mu B_{rst}) = |u^*(B_{ijk} +\mu B_{rst})u| = 1$.
By Lemma \ref{le3}, one conclude that
$U^*B_{rst}U$ has the form $\begin{pmatrix} 0& 0 \cr 0 & * \end{pmatrix}$,
where $U$ is a unitary matrix with $u$ as its first column,
and hence $B_{rst} u = 0$.

\medskip
Next, suppose $\delta_{ir}+\delta_{js}+\delta_{kt}=1$, say $i = r$.
By the previous case, $B_{isk} u =  B_{ijt} u = 0$.
Then for any $\mu \in \IC$ with $|\mu| \le 1$,
\begin{eqnarray*} %\label{ck-2}
1 &=&  w(B_{ijk} + B_{isk} + \mu (B_{ijt} + B_{ist})) \cr
&\ge&  |u^*(B_{ijk} + B_{isk} + \mu (B_{ijt} + B_{ist}))u| \cr
&=& |u^*B_{ijk}u + \mu u^*B_{ist}u|
= |e^{i\theta} + \mu \xi|.
\end{eqnarray*}
It follows that $\xi = u^*B_{ist} u = 0$ and hence
$w(B_{ijk} + B_{isk} + \mu (B_{ijt} + B_{ist})) = |u^*B_{ijk} u| = 1$.
By Lemma \ref{le3}, we conclude that $(B_{ijt} + B_{ist})u = 0$ and thus, $B_{ist} u = 0$.

\medskip
Finally, suppose $\delta_{ir}+\delta_{js}+\delta_{kt}=0$.
By the previous cases,
$B_{isk} u = B_{rjk} u = B_{rsk} u = B_{ijt} u = B_{ist} u = B_{rjt} u = 0$.
Then for any $\mu \in \IC$ with $|\mu| \le 1$,
\begin{eqnarray*} %\label{ck-2}
1 &=&  w((B_{ii}+B_{rr})\otimes (B_{jj}+B_{ss})\otimes (B_{kk}+\mu B_{tt})) \cr
&\ge&  |u^*((B_{ii}+B_{rr})\otimes (B_{jj}+B_{ss})\otimes (B_{kk}+\mu B_{tt}))u| \cr
&=& |u^*B_{ijk}u + \mu u^*B_{rst}u|
= |e^{i\theta} + \mu \xi|.
\end{eqnarray*}
%\begin{eqnarray*} %\label{ck-2}
%1 &=&  w(B_{ijk} + B_{isk} + B_{rjk} + B_{rsk} + \mu (B_{ijt} + B_{ist} + B_{rjt} + B_{rst})) \cr
%&\ge&  |u^*(B_{ijk} + B_{isk} + B_{rjk} + B_{rsk} + \mu (B_{ijt} + B_{ist} + B_{rjt} + B_{rst}))u| \cr
%&=& |u^*B_{ijk}u + \mu u^*B_{rst}u|
%= |e^{i\theta} + \mu \xi|.
%\end{eqnarray*}
Then by a similar argument, one conclude that $\xi = 0$ and $B_{rst} u = 0$.
\qed

\vspace{0,4cm}

\noindent
{\bf Claim 5.}
Suppose $(i,j,k) \ne (r,s,t)$. If $u_{ijk}, u_{rst} \in \IC^{mn}$
are two unit vectors such that \linebreak
$|u_{ijk}^*B_{ijk} u_{ijk}| = |u_{rst}^*B_{rst} u_{rst}| = 1$,
then $u_{ijk}^* u_{rst} = 0$.

\vspace{0,2cm}

\it Proof. \rm Suppose $u_{rst} =  \alpha u_{ijk} + \beta v$
with $\alpha = u_{ijk}^* u_{rst}$ and $\beta = v^*u_{rst}$,
where $v$ is a unit vector orthogonal to $u_{ijk}$.
Note that $|\alpha|^2 + |\beta|^2 = 1$.
By Claim 1, $u_{ijk}^*B_{rst} = 0$ and $B_{rst} u_{ijk} = 0$ and so
$$1 = |u_{rst}^*B_{rst} u_{rst} |
= |\beta|^2 |v^* B_{rst} v| \le |\beta|^2 w(B_{rst}) = |\beta|^2 \le 1.$$
Thus, $|\beta| = 1$ and hence $u_{ijk}^* u_{rst} = \alpha = 0$.
\qed

\vspace{0,4cm}

\noindent
{\bf Claim 6.}
Let $$U=[u_{111}~\cdots~u_{11n_3}~u_{121}~\cdots~u_{12n_3}~\cdots~u_{1n_2n_3}
~\cdots~u_{211}~\cdots~u_{2n_2n_3}~\cdots~u_{n_111}~\cdots~u_{n_1n_2n_3}].$$
Then $U^*U = I_N$ and
$U^*B_{ijk}U = \xi_{ijk} (E_{ii}\otimes E_{jj}\otimes E_{kk})$ for all
$1\le i\le n_1, 1\le j\le n_2$, $1\le k\le n_3$.

\vspace{0,2cm}
\it Proof. \rm By Claim 2, $\{u_{ijk}: 1\le i\le n_1, 1\le j\le n_2, 1\le k \le n_3\}$ forms an orthonormal basis and thus $U^*U = I_{N}$.
Next by Claim 1, $u_{rst}^* B_{ijk} u_{ \ell pq} = 0$ for all $(r,s,t)$ and $( \ell, p,q)$,
except the case when $(r,s,t) = ( \ell, p,q) = (i,j,k)$.
Therefore, the result follows.
\qed

\vspace{0,2cm}

%%%%%%%%%%%%%%%%%%%%%%%%%%%%%%%%%%%%%%%%%%%%%%%%%%%%%%%%%%%%%%%%%%%%%%%%%%%%%%%%%%%%%%%%%%%%%%%%%%%%%%%%%%%%%%%%%%%%%%%%%%%%%%%%%%%%%%%%%%%%%%%%%%%%%%%%%%%%%%%%%%%%%%%%%%%%%%%%%%%%%

Similarly, for any unitary $X\in M_{n_1}$, there exists some unitary $U_X$ and some complex units $\mu_{i_1i_2\cdots i_m}$ such that
$$\phi(XE_{i_1i_1}X^* \otimes E _{i_2i_2}\otimes \cdots \otimes E_{i_mi_m}) = \mu_{i_1i_2\cdots i_m} U_X(XE_{i_1i_1}X^* \otimes E _{i_2i_2}\otimes \cdots \otimes E_{i_mi_m})U_X^*$$
for all $1\le i_k\le n_k$ with $1\le k\le m$. We see that $\phi(XE_{i_1i_1}X^* \otimes E _{i_2i_2}\otimes \cdots \otimes E_{i_mi_m})$ is a rank one matrix with numerical radius one. If  $\gamma > 0$, then $$w(\phi((XE_{i_1i_1}X^*+\gamma I_{n_1})\otimes E _{i_2i_2}\otimes \cdots \otimes E_{i_mi_m})) = 1+\gamma.$$
Thus, $\phi(XE_{i_1i_1}X^* \otimes E _{i_2i_2}\otimes \cdots \otimes E_{i_mi_m})$
has the form $R \otimes E _{i_2i_2}\otimes \cdots \otimes E_{i_mi_m})$
for some $R \in M_{n_1}$. Since this is true for any unitary $X\in M_{n_1}$, we have $$\phi(A\otimes E _{i_2i_2}\otimes \cdots \otimes E_{i_mi_m}) = \varphi_{ i_2\cdots i_m}(A)\otimes E _{i_2i_2}\otimes \cdots \otimes E_{i_mi_m}$$
for all Hermitian matrices $A\in M_{n_1}$ and some linear map $\varphi_{ i_2\cdots i_m}$. Clearly, $ \varphi_{ i_2\cdots i_m}$ preserves numerical radius and, hence, has the form
$$A \mapsto \xi_{ i_2\cdots i_m} W_{ i_2\cdots i_m}AW_{ i_2\cdots i_m}^* \quad \hbox{ or } \quad A \mapsto \xi_{ i_2\cdots i_m}  W_{ i_2\cdots i_m}A^tW_{ i_2\cdots i_m}^*$$
for some complex unit $\xi_{ i_2\cdots i_m}$ and unitary $W_{ i_2\cdots i_m} \in M_{n_1}$. In particular, $\varphi_{ i_2\cdots i_m}(I_{n_1})=\xi_ { i_2\cdots i_m} I_{n_1}$ and
$\phi(I_N) = I_{n_1}\otimes D_1$ for some diagonal matrix $D_1\in M_{n_2\cdots n_m}$.

Given $2\leq k\leq m$ and using the same arguments as above, one can show that $\phi(I_N) = D_{k1}\otimes I_{n_k}\otimes D_{k2}$ for some diagonal matrix
$D_{k1}\in M_{n_1\cdots n_{k-1}}$ and $D_{k2}\in M_{n_{k+1}\cdots n_m}$. Since $$\phi(I_N) = I_{n_1}\otimes D_1 =D_{k1}\otimes I_{n_k}\otimes D_{k2}~{\rm for}~k=2,\ldots,m,$$
we conclude that  $\phi(I_N)=\xi I_N$ for some complex unit $\xi$. For the sake of the simplicity, let us assume that $\phi(I_N)= I_N$. Then
$$\phi(E_{i_1i_1} \otimes E _{i_2i_2}\otimes \cdots \otimes E_{i_mi_m}) =  E_{i_1i_1} \otimes E_{i_2i_2}\otimes \cdots \otimes E_{i_mi_m} $$
for all $1\le i_k\le n_k$ with $1\le k\le m$.

For any Hermitian matrix  $\A\in M_N$, suppose its spectral decomposition  is $$\A=X_1D_1X_1^*\otimes \cdots \otimes X_mD_mX_m^*.$$
Repeating the above argument  and using the assumption $\phi(I_{N})= I_{N}$, we can conclude that there exists a unitary matrix $U_{X_1,\ldots,X_m}$
such that $$\phi(X_1E_{i_1i_1}X_1^*\otimes \cdots\otimes X_mE_{i_mi_m}X_m^*)=U_{X_1,\ldots,X_m}(X_1E_{i_1i_1}X_1^*\otimes \cdots\otimes X_mE_{i_mi_m}X_m^*)U_{X_1,\ldots,X_m}^*$$ for all $1\le i_k\le n_k$ with $1\le k\le m$. By linearity, we have $$\phi(\A)=U_{X_1,\ldots,X_m}(\A)U_{X_1,\ldots,X_m}^*.$$
So, $\phi$ maps Hermitian matrices to Hermitian matrices and preserves numerical range on the tensor product of Hermitian matrices.
 By Theorem \ref{T5},  it has the asserted form on Hermitian matrices and, hence, on all matrices in $M_{n_1\cdots n_m}$.
If $n_i, n_j \ge 3$, we observe the matrices $A_i = X \oplus 0_{n_i-3}$ and  $A_j = X \oplus 0_{n_j-3}$, where $X$ is defined as in
Example \ref{Ex1}, and $A_k=E_{11}\in M_{n_k}$ for $k\ne i,j$, to conclude that $\varphi_i$ and $\varphi_j$ should both be the identity map,
or both be the transpose map. The proof is completed.
\end{proof}

\bigskip
\noindent{\bf Acknowledgment}

\noindent
This research was supported by a Hong Kong GRC grant PolyU 502411 with Sze as the PI.
The grant also supported the post-doctoral fellowship of Huang
and the visit of Fo\v{s}ner to the Hong Kong Polytechnic University in the summer of 2012.
She gratefully acknowledged the support and kind hospitality from the host university.
Fo\v sner was supported by the bilateral research program between Slovenia and US (Grant No. BI-US/12-13-023).
Li was supported by a GRC grant and a USA NSF grant; this research was done when he was
a visiting professor of the University of Hong Kong in the spring of 2012; furthermore,
he is an honorary professor of Taiyuan University of Technology (100 Talent Program scholar),
and an honorary professor of the  Shanghai University.

\baselineskip16pt

%%%%%%%%%%%%%%%%%%%%%%%%%%%%%%%%%%%%%%%%%%%%%%%%%%%%%%%%%%%%%%%%%%%%%%%%%%%%%%%%%%%%%%%%%%%%%%%%%%%%%%%%%%%%%%%%%%%%%%%%%%%%%%%%%%%%%%%%%%%%%%%%%%%%%%%%%%%%%%%%%%%%%%%%%%%%%%%%%%%%%%%%%%%%%%%%%%%%%%%%%%%%%%%%%%%%%%%%%%%%%%%%%%%%%%%%%%%%%%%%%%%%%%%%%%%%%%%%%%%%%%%%%%%%%%%%%%%%%%%%%%%%%%%%%%%%%%%%%%%%%%%%%%%%%%%%%%%%%%%%%%%%%%%%%%%%%%%%%%%%%%%%%%%%%%%%%%%%%%%%%%%

%%%%%%%%%%%%%%%%%%%%%%%%%%%%%%%%%%%%%%%%%%%%%%%%%%%%%%%%%%%%%%%%%%%%%%%%%%%%%%%%%%%%%%%%%%%%%%%%%%%%%%%%%%%%%%%%%%%%%%%%%%%%%%%%%%%%%%%%%%%%%%%%%%%%%%%%%%%%%%%%%%%%%%%%%%%%%%%%%%%%%%%%%%%%%%%%%%%%%%%%%%%%%%%%%%%%%%%%%%%%%%%%%%%%%%%%%%%%%%%%%%%%%%%%%%%%%%%%%%%%%%%%%%%%%%%%%%%%%%%%%%%%%%%%%%%%%%%%%%%%%%%%%%%%%%%%%%%%%%%%%%%%%%%%%%%%%%%%%%%%%%%%%%%%%%%%%%%%%%%%%%%

\end{document}